\documentclass{amsart}

\usepackage{latexsym,amsmath,amssymb,amsfonts,amscd,graphics}
\usepackage[alphabetic,lite,nobysame]{amsrefs}
\usepackage[mathscr]{eucal}
\usepackage{mathrsfs}
\usepackage{tikz-cd}
\usepackage{hyperref}
\usepackage{cleveref}
\usepackage{thmtools}
\usepackage{mathtools}
\usepackage{multicol}
\usepackage[all]{xy}
\usepackage{soul}

\hypersetup{
	colorlinks=true,
	linkcolor=blue,
	filecolor=blue,      
	urlcolor=cyan,
	citecolor=magenta}

\numberwithin{equation}{section}
\theoremstyle{plain}
\newtheorem{theorem}[equation]{Theorem}

\newtheorem{lemma}[equation]{Lemma}
\newtheorem{proposition}[equation]{Proposition}
\newtheorem{corollary}[equation]{Corollary}

\newtheorem{conj}[equation]{Conjecture}
\theoremstyle{remark}
\newtheorem{remark}[equation]{Remark}

\theoremstyle{definition}
\newtheorem{definition}[equation]{Definition}

\newcommand{\bP}{\mathbb{P}}

\newcommand{\bR}{\mathbb{R}}

\newcommand{\bZ}{\mathbb{Z}}

\newcommand{\bC}{\mathbb{C}}

\newcommand{\calA}{\mathcal{A}}
\newcommand{\calC}{\mathcal{C}}

\newcommand{\calM}{\mathcal{M}}

\newcommand{\calP}{\mathcal{P}}
\newcommand{\calI}{\mathcal{I}}

\newcommand{\calD}{\mathcal{D}}

\newcommand{\calW}{\mathcal{W}}

\newcommand{\Aut}{\mathrm{Aut}}
\newcommand{\Orth}{\mathrm{O}}
\newcommand{\Bir}{\mathrm{Bir}}
\newcommand{\Bl}{\mathrm{Bl}}

\newcommand{\Hilb}{\mathrm{Hilb}}

\newcommand{\Mov}{\mathrm{Mov}}
\newcommand{\Cone}{\mathrm{Cone}}

\newcommand{\im}{\mathrm{Im}}
\newcommand{\id}{\mathrm{Id}}

\newcommand{\Gr}{\mathrm{Gr}}

\newcommand{\pex}{\mathrm{pex}}
\newcommand{\flop}{\mathrm{flop}}

\newcommand{\Nef}{\mathrm{Nef}}

\newcommand{\NS}{\mathrm{NS}}

\newcommand{\Mon}{\mathrm{Mon}}

\newcommand{\git}{/\kern-0.2em/}

\newcommand{\Amp}{\mathrm{Amp}}

\newcommand{\bir}{\simeq_{\mathrm{bir}}}

\title[]{Cubic fourfolds with birational Fano varieties of lines}

\author{Corey Brooke}
\address{Department of Mathematics and Statistics, Carleton College, Northfield, MN 55057}
\email{cbrooke@carleton.edu}

\author{Sarah Frei}
\address{Department of Mathematics, Dartmouth College, Kemeny Hall, Hanover, NH 03755}
\email{sarah.frei@dartmouth.edu}

\author{Lisa Marquand}
\address{Courant Institute,
  251 Mercer Street,
  New York, NY 10012, USA}
\email{lisa.marquand@nyu.edu}

\begin{document}

\maketitle

\begin{abstract}
We give several examples of pairs of non-isomorphic cubic fourfolds whose Fano varieties of lines are birationally equivalent (and in one example isomorphic). Two of our examples, which are special families of conjecturally irrational cubics in $\calC_{12}$, provide new evidence for the conjecture that Fourier-Mukai partners are birationally equivalent. We explore how various notions of equivalence for cubic fourfolds are related, and we conjecture that cubic fourfolds with birationally equivalent Fano varieties of lines are themselves birationally equivalent.
\end{abstract}

\section{Introduction}

In recent decades, smooth cubic fourfolds have been studied up to a few different forms of equivalence, especially
\begin{enumerate}
    \item[({\bf{BE}})] birational equivalence and
    \item[({\bf{FM}})] Fourier-Mukai partnership, i.e. equivalence of the Kuznetsov components $\calA_X$ of the derived categories of the two cubics, as defined in \cite{Kuznet}.
\end{enumerate}

In recent work with Qin, we identified two non-isomorphic cubic fourfolds whose Fano varieties of lines are birationally equivalent (see \cite[Theorem 5.1]{BFMQ24}). This suggests studying a third form of equivalence for cubic fourfolds:
\begin{enumerate}
    \item[({\bf{BF}})] birational equivalence of the Fano varieties of lines.
\end{enumerate}

It is natural to ask how these notions of equivalence are related. Indeed, in view of the famous rationality conjectures for cubic fourfolds (see \cite{hassett}, \cite{Kuznet}, \cite{addington}, \cite{MR4292740}),
Huybrechts conjectured the following:

\begin{conj}[{\cite[Conjecture 7.3.21]{Huybrechtscubicsbook}}]\label{conj:FMimpliesBE}
Suppose $X$ and $X'$ are smooth cubic fourfolds with $\calA_{X}\simeq \calA_{X'}$. Then $X\bir X'$. That is, $\mathrm{(}${\bf{FM}}$\mathrm{)}$ implies $\mathrm{(}${\bf{BE}}$\mathrm{)}$.    
\end{conj}

A very general cubic fourfold, i.e. $X\not\in \cup\, \calC_d$, has no non-trivial Fourier-Mukai partners by \cite{Huy17}. The conjecture becomes interesting for $X\in \calC_d,$  with some evidence provided by Fan and Lai \cite{FL24} who verified the conjecture for a very general cubic in $\calC_{20}$. On the other hand, it is well known that ({\bf{BE}}) does not imply ({\bf{FM}}), for example by considering two very general cubics in $\calC_{14}$.

By studying two special families of cubics in $\calC_{12}$ (see Theorems \ref{thm: nonsyz} and \ref{theorem:C8C12}), we provide two new corroboratory examples for \Cref{conj:FMimpliesBE}. The first is the family $\calC_{nonsyz}$, parametrizing cubics containing a non-syzygetic pair of cubic scrolls, introduced in \cite{BFMQ24}. The second is a component $\calC_{M_1}$ of $\calC_8\cap \calC_{12}$, studied in \cite{BolognesiRusso}, \cite{yang2021lattice}, and \cite{C8C12}, parametrizing cubics containing a plane and a cubic scroll intersecting in a point. For a general cubic fourfold $X$ in either of these two families, we prove the existence of another non-isomorphic cubic fourfold $X'$ (contained in the same family) which is both a Fourier Mukai partner of $X$, and is birationally equivalent to $X$. We also find that $X$ and $X'$ have birationally equivalent Fano varieties of lines in both cases.

Motivated by these families of cubics, as well as by our study of the Fano variety of lines on the general member of $\calC_{20}$, we suggest the following new conjecture:

\begin{conj}\label{conj:BFimpliesBE}
    Suppose $X$ and $X'$ are smooth cubic fourfolds with birationally equivalent Fano varieties of lines. Then $X\bir X'$. That is,  $\mathrm{(}${\bf{BF}}$\mathrm{)}$ implies $\mathrm{(}${\bf{BE}}$\mathrm{)}$.
\end{conj}

Conversely, the fact that ({\bf{BE}}) does not imply ({\bf{BF}}) is trivial: for example, one can again consider two very general cubics in $\calC_{14}$, which are both rational. Our examples in Sections~\ref{sec:nonsyz} through~\ref{sec:546} are conjecturally irrational.

The Fano variety of lines on a cubic fourfold is a smooth hyperk\"ahler fourfold of K3$^{[2]}$-type \cite{BD85}. 
It is known that two hyperk\"ahler fourfolds of K3$^{[2]}$-type that are birational are derived equivalent \cite{kawamata}\cite{namikawa} (see also \cite{MS24} for a more general statement). This motivates a fourth notion of equivalence between cubic fourfolds:
\begin{enumerate}
    \item[({\bf{DF}})] derived equivalence of the Fano varieties of lines.
\end{enumerate}

It has been conjectured that ({\bf{FM}}) implies ({\bf{DF}}) \cite[Remark 7.3.28(iii)]{Huybrechtscubicsbook}, for which we provide corroboratory evidence (see Theorems~\ref{thm: nonsyz}, \ref{theorem:20}, and \ref{theorem:C8C12}). On the other hand, we prove that neither ({\bf{DF}}) nor ({\bf{FM}}) implies ({\bf{BF}}) by considering the very general member of $\calC_{546}$ (see \Cref{theorem:546}). 
The fact that ({\bf{BE}}) does not imply ({\bf{DF}}) follows from the fact that hyperk\"ahler fourfolds have finitely many Fourier-Mukai partners \cite[Theorem~9.4]{Beckmann}.

The following diagram summarizes the various implications, counterexamples, and conjectures discussed above:

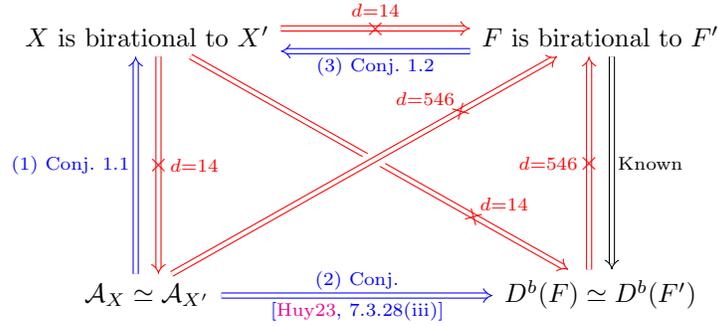
\begin{equation*}
    \begin{tikzcd}[column sep=large, row sep=large, arrows={crossing over}]
         X \text{ is birational to } X' \arrow[rr, Rightarrow, "\times" marking, yshift=1ex, color=red, "d=14"{yshift=2pt}]  \arrow[ddd, Rightarrow, "\times" marking, xshift=1ex, color=red, "d=14"{xshift=2pt}] \arrow[dddrr, Rightarrow, xshift=1.2ex, color=red, "\times" marking, near end, "d=14"] && 
         \arrow[ll, Rightarrow, yshift=-1ex, color=blue, "\text{(3) Conj.~\ref{conj:BFimpliesBE}}"]   F \text{ is birational to } F' \arrow[ddd, Rightarrow, xshift=1ex, "\text{Known}"]  \\
         && \\
         && \\
        \calA_X \simeq \calA_{X'}  \arrow[rr, Rightarrow, yshift=0ex, color=blue, "\text{(2) Conj.}", "\text{\cite[7.3.28(iii)]{Huybrechtscubicsbook}}" '] \arrow[uuu, Rightarrow, xshift=-1ex, color=blue, "\text{(1) Conj.~\ref{conj:FMimpliesBE}}"] \arrow[uuurr, Rightarrow, xshift=-1.2ex, color=red, "\times" marking, near end, "d=546"] &&  
        D^b(F)\simeq D^b(F') \arrow[uuu, Rightarrow, xshift=-1ex, color=red, "\times" marking, "d=546"{xshift=-2pt}] \\
    \end{tikzcd}
\end{equation*}
We add to the body of supporting evidence for each numbered arrow, proving that the very general member of each family listed below satisfies both conditions.
\begin{enumerate}
    \item[(1)]  $\calC_{nonsyz}$ and $\calC_{M_1}$
    \item[(2)]  $\calC_{nonsyz}$, $\calC_{20}$, and $\calC_{M_1}$
    \item[(3)]  $\calC_{nonsyz}$, $\calC_{20}$, and $\calC_{M_1}$
\end{enumerate}

We are not aware of any general results about whether ({\bf{DF}}) implies ({\bf{FM}}), ({\bf{DF}}) implies ({\bf{BE}}), or ({\bf{BF}}) implies ({\bf{FM}}), although the first holds for $X\in C_d$ with $d$ satisfying $(**)$, by \cite[Theorem~1]{AddingtonTwoRatConjs} and \cite[Corollary~9.7]{Beckmann}. Moreover, the families we study in Sections~\ref{sec:nonsyz} through~\ref{sec:C8C12} give evidence for ({\bf{BF}}) implies ({\bf{FM}}); this implication is particularly relevant to the conjectures discussed above. Indeed, showing that ({\bf{BF}}) implies ({\bf{FM}}) along with proving Conjecture~\ref{conj:FMimpliesBE} would imply Conjecture~\ref{conj:BFimpliesBE}. 
Determining whether any of the equivalence relations ({\bf{BF}}), ({\bf{FM}}), or ({\bf{DF}}) considered here implies ({\bf{BE}}) seems quite difficult, since no pair of cubic fourfolds is known to be birationally inequivalent.
\smallskip

\subsection*{Outline} In \Cref{sec: prelims}, we recall the definition of special cubic fourfolds, rationality conjectures, and necessary results on the birational geometry of hyperk\"ahler fourfolds and Fano variety of lines. We also introduce Gushel-Mukai fourfolds and double EPW sextics. In \Cref{sec:nonsyz}, we study equivalences discussed above for cubic fourfolds contained in $\mathcal{M}_{nonsyz},$ i.e. containing a non-syzygetic pair of cubic scrolls. In \Cref{sec:20} we study cubic fourfolds contained in $\calC_{20}$. In \Cref{sec:C8C12}, we study cubic fourfolds containing both a plane and a cubic scroll, focusing on $\calC_{M_1}$, which is one of three families in $\calC_8\cap \calC_{12}$. Finally in \Cref{sec:546}, we study cubic fourfolds in $\calC_{546}$. Our choice to consider $\calC_{546}$ is motivated by numerical conditions guaranteeing that $F$ is birational to the Hilbert square of a K3 with a nontrivial Fourier-Mukai partner (see \Cref{remark:546numerics} for details).

\subsection*{Acknowledgements}
We would like to thank Asher Auel for offering helpful comments on the first version of this paper. We would like to thank Nick Addington for pointing out an error in an earlier version.

S.F. was supported in part by NSF grant DMS-2401601. Computations of orthogonal groups were done in Magma \cite{Magma}.

\section{Preliminaries}\label{sec: prelims}

\subsection{Special cubic fourfolds}\label{subsec:special}
For a smooth cubic fourfold $X$, let
\[
A(X)=H^4(X,\bZ)\cap H^{2,2}(X,\bC)
\]
be the lattice of algebraic cycles and $\eta_X$ the square of the hyperplane class. We define
\[
H^4(X,\bZ)_{prim}=\langle\eta_X\rangle^\perp\subset H^4(X,\bZ)
\]
and
\[
A(X)_{prim}=H^4(X,\bZ)_{prim}\cap A(X).
\]
We have that $A(X)_{prim}=0$ for a very general cubic fourfold, but there are countably many divisors $\calC_d$ in the moduli space of cubic fourfolds which parametrize those with $A(X)_{prim}\neq0$. An $X\in\calC_d$ is such that $\eta_X\in K\subset A(X)$ where $K$ is a primitive sublattice of rank two, and the discriminant of the intersection form on $K$ is $d$. We call $K$ a labeling of $X$.

Since the integral Hodge conjecture holds for cubic fourfolds, every class in $A(X)$ is algebraic, and cubics in $\calC_d$ for some $d$ contain algebraic surfaces not homologous to complete intersections \cite[Theorem 1.4]{voisinhodge}. We give a few examples that appear in this paper:
\begin{itemize}
    \item every member of $\calC_8$ contains a plane \cite{voisintorelli},
    \item the general member of $\calC_{12}$ contains a cubic scroll \cite{hassett-thesis}, and
    \item the general member of $\calC_{20}$ contains a Veronese surface \cite{hassett-thesis}.
\end{itemize}
Whereas neither a plane nor a Veronese surface deform in a cubic fourfold, the general member of $\calC_{12}$ in fact contains two nets of cubic scrolls, as mentioned in \cite{hassett-thesis} and expanded on in \cite{flops}. More specifically, if $T\subset X$ is a cubic scroll, $\langle T \rangle$ is the hyperplane spanned by $T$, and $Q$ is a quadric containing $T$, then $X\cap Q\cap \langle T \rangle=T\cup T^\vee$ where $T^\vee$ is another cubic scroll.

One can articulate rationality conjectures for cubic fourfolds in terms of labelings. By \cite{hassett-thesis}, the discriminants $d$ for which $\calC_d$ is nonempty are those satisfying
\[
(*):\;d>6\text{ and }d\equiv0\text{ or }2\bmod 6.
\]
Cubics in $\calC_d$ satisfying
\[
(**):\;d\text{ satisfies }(*),\text{ and is not divisible by 4, 9, or any odd prime }p\equiv2\bmod3
\]
have an associated K3 surface $S$, meaning $\widetilde{H}(\calA_X,\bZ)\cong\widetilde{H}(S,\bZ)$. All cubic fourfolds known to be rational belong to $\calC_d$ for some $d$ satisfing $(**)$, leading to the famous conjecture that cubic fourfolds are rational if and only if they admit a labelling of disrciminant $d$ satisfying $(**)$. Kuznetsov conjectured in \cite{kuzcubic} that $X$ is rational if and only if $\calA_X\simeq D^b(S)$ for some K3 surface $S$, which was shown in \cite{AddingtonTwoRatConjs} and \cite{MR4292740} to coincide with the previous conjecture.

A competing conjecture proposed by Galkin and Shinder in \cite{galkinshinder} is that a cubic fourfold $X$ is rational if and only if its Fano variety of lines is birational to $\Hilb^2(S)$ for some K3 surface $S$. This condition was shown in \cite{AddingtonTwoRatConjs} to be equivalent to $X$ admitting a labeling of discriminant $d$ satisfying 
\[
(***):d\text{ satisfies }(*),\text{ and }d=\frac{2n^2+2n+2}{a^2}\text{ for some }a,n\in\bZ.
\]
Moreover, $(***)$ implies $(**)$, and this condition is stricter.

\subsection{Birational geometry of hyperk\"ahler fourfolds}
We briefly recall how one can study the birational geometry of a hyperk\"ahler fourfold $F$ from the movable cone. Throughout, $F$ is a projective hyperk\"ahler manifold of K3$^{[2]}$ type.
The second cohomology $H^2(F,\bZ)$ is equipped with an integral, symmetric, nondegenerate bilinear form $q_F$, namely the Beauville-Bogomolov-Fujiki (BBF) form. For $v, w\in H^2(F,\bZ)$, we denote by $v^2=q_F(v)$ and by $v\cdot w= q_F(v,w)$. The divisibility $\mathrm{div}(v)$ of $v$ is the least positive generator of the ideal $v\cdot H^2(F,\bZ)\subset \bZ$. 

We denote by $\Mon^2_{Hdg}(F)$ the subgroup of monodromy operators in $\Mon^2(F)$ that preserve the Hodge structure. Recall that $\Mon^2_{Hdg}(F)=W_{Exc}\rtimes W_{Bir}$, where $W_{Exc}$ is the subgroup generated by reflections determined by stably prime exceptional divisors, and $W_{Bir}$ is the subgroup of monodromy operators induced by birational transformations. 

\begin{theorem}\cite[Theorem 1.3]{markman}\label{thm:torelli}
    Let $F$ be a projective hyperk\"ahler manifold.
    Let $g\in \Mon^2_{Hdg}(F)$. Then there exists $f\in \Bir(F)$ such that $f^*=g$ if and only if $g^*\Mov(F)=\Mov(F)$. Further, $f\in \Aut(F)$ if and only if $g^*\Amp(F)=\Amp(F)$.
\end{theorem}

Let $\overline{\mathrm{Pos}(F)}$ be the component of the cone $\{x\in \NS(F)\otimes \bR \mid x^2 \geq 0\}$ that contains an ample class. We denote by $\Mov(F)\subset \overline{\mathrm{Pos}(F)}$ the (closed) convex cone generated by classes of line bundles on $F$ whose base locus has codimension at least $2$.
We define the following set of divisors:
\begin{align*}
    \calW_{\pex}&\coloneqq \{\rho \in \NS(F) \mid \rho^2=-2 \}\\
    \calW_{\flop}&\coloneqq \{\rho\in \NS(F) \mid \rho^2=-10, \textrm{div}(\rho)=2\}.
\end{align*}
A divisor $\rho\in \calW_{\pex}\cup \calW_{\flop}$ is called a wall divisor \cite[Definition 1.2, Proposition 2.12]{MR3423735}, and $\rho\in \calW_{\pex}$ is moreover referred to as a stably prime exceptional divisor. The following structure theorem for $\Mov(F)$ and $\Amp(F)$ in the case of hyperk\"ahler manifolds of K3$^{[2]}$ type combines the results of Markman \cite[Lemma 6.22, Proposition 6.10, 9.12, Theorem 9.17]{markman} and  Bayer, Hassett and Tschinkel \cite[Theorem 1]{BHT}.

\begin{theorem}\cite[Theorem 3.16]{debarre2020hyperkahler}\label{thm: Mov and Amp}
    Let $F$ be a hyperk\"ahler manifold of K3$^{[2]}$ type.

    \begin{enumerate}
        \item The interior of $\Mov(F)$ is the connected component of $$\overline{\mathrm{Pos}(F)}\setminus \bigcup_{\rho\in \calW_{\pex}} \rho^\perp$$ that contains the class of an ample divisor.
        \item The ample cone $\Amp(F)$ is the connected component of 
        $$\overline{\mathrm{Pos}(F)}\setminus \bigcup_{\rho\in \calW_{\pex}\cup \calW_{\flop}} \rho^\perp$$ that contains the class of an ample divisor.
    \end{enumerate}
\end{theorem}

Each chamber of $\Mov(F)$ corresponds to the pullback of $\Nef(F')$ of a birational model $F'$ of $F$, and $\Mon^2_{Hdg}(F)= W_{Exc}\rtimes W_{Bir}$ acts transitively on the chambers of $\bigcup_{\rho\in \calW_{\pex}} \rho^\perp$ \cite[Lemma 5.11 (5)]{markman}.

\subsection{Fano varieties of lines}\label{subsec:Fano} The Fano variety of lines $F$ on a cubic fourfold is a hyperk\"ahler manifold of K3$^{[2]}$ type. 
The Abel-Jacobi map $\alpha\colon H^4(X,\bZ)\rightarrow H^2(F,\bZ)$ restricts to an isomorphism $H^4(X,\bZ)_{prim}\xrightarrow{\sim} H^2(F,\bZ)_{prim}$ which is a Hodge isometry (up to sign) under the respective quadratic forms. Explicitly, we have $-x\cdot y=\alpha(x)\cdot\alpha(y)$.

The variety $F$ is naturally equipped with the Pl\"ucker polarization $g\in \NS(F)$, with $g^2=6$ and divisibility $2$. Recall that $\calM_6^{(2)}$ is the moduli space of hyperk\"ahler fourfolds of K3$^{[2]}$ type equipped with a polarization $L$ with $L^2=6$ and divisibility $2$. There is a period map:
$p\colon \calM_6^{(2)}\rightarrow\calD_6^{(2)},$ where 
$\calD_6^{(2)}$ is the relevant period domain (see for instance \cite[Appendix B]{debarre2020hyperkahler}).

\begin{proposition}\label{prop: existence of cubic}
    Let $F$ be the Fano variety of lines for a cubic fourfold $X$ with Pl\"ucker polarization $g$. Suppose that
    \begin{enumerate}
        \item there exists $h\in \Mov(F)$ such that $h^2=6$ and $\mathrm{div}(h)=2$, and
        \item there exists $\phi\in \Orth(\NS(F))$ such that $\phi(g)=h$.
    \end{enumerate}
     Then there exists a smooth cubic fourfold $X'$ with Fano variety of lines $F'$ and a birational equivalence $\pi:F'\dashrightarrow F$ such that $\pi^*(h)$ is the Pl\"ucker class. Moreover, $X\cong X'$ if and only if there exists $\phi\in \Bir(F)$ such that $\phi^*(h)=g$.    
     \end{proposition}
\begin{proof}
    By assumptions (1), $h$ is ample on some birational model $F'$ of $F$. The pairs $(F,g)$ and $(F',h)$ define points in the moduli space $\calM_6^{(2)}$, which are distinct if and only if the isometry $\phi\in \Orth(\NS(F))$ of assumption (2) is not induced by a birational isomorphism of $F$. We have the commutative diagram of period maps, where $\calD_6^{(2)}$ is the relevant period domain.

    \[\xymatrix{ \calM_{cub} \ar[rr]^{F} \ar[dr]_{p'} & & \calM_6^{(2)} \ar[dl]^{p} \\ 
            & \calD_6^{(2)} &
    },\]  
    Let $D_{6,i}^{(2)}\subset \calD_6^{(2)}$ be the Heegner divisor which is the image in $\calD_6^{(2)}$ of a hyperplane section of the form $v^\perp$ with $v^2=-i$.
    By \cite{laz09} and \cite{Loo}, the union $D_{6,2}^{(2)}\cup D_{6,6}^{(2)}$ is the complement of the image of the period map $p'$. We claim that $p(F',h)$ lies in the complement of $D_{6,2}^{(2)}\cup D_{6,6}^{(2)}.$  If not, then there exists a $v\in h^\perp\cap \NS(F')$ with $v^2= -2$ or $-6$ and divisibility $2$. Then since $\NS(F')\cong \NS(F)$, we can view $\phi^{-1}\in \Orth(\NS(F'))$, and $\phi^{-1}(v)\cdot g=0;$ this is a contradiction since $p'(X)=p(F,g)$ lies in the complement of $D_{6,2}^{(2)}\cup D_{6,6}^{(2)}$.

It follows that there exists a smooth cubic fourfold $X'$ such that $p'(X')=p(F',h)$. In particular, $(F',h)$ is isomorphic to $F(X')$ with the Pl\"ucker polarization given by $h$. Further, if $(F,g)$ and $(F',h)$ are distinct points of $\calM_6^{(2)}$, then it follows that $X'\not\cong X'$ by the Global Torelli theorem for cubic fourfolds \cite{voisintorelli}. 
\end{proof}

\subsection{Gushel-Mukai fourfolds}\label{subsec: GM4fold}

A Gushel-Mukai fourfold (GM fourfold) $Z$ is a smooth transverse intersection of the form
\[
    Z \coloneqq \Gr(2,V_5) \cap \bP^8\cap Q\subset \bP\left(\bigwedge\nolimits^2V_5\right),
\]
    where $V_5$ is a five dimensional vector space, $\Gr(2,V_5)$ is given the Pl\"ucker embedding, $\bP^8$ is a linear subspace, and $Q$ is a quadric. Given a GM fourfold $Z$, one can construct two Eisenbud-Popescu-Walter sextics (EPW sextics), $W_A$ and its dual $W_{A^\perp}$, constructed from dual Lagrangian subspaces $A\subset \wedge^3V_6$ and $A^\perp\subset \wedge^3V_6^\vee.$ Taking an appropriate double cover of $W_A$ (resp. $W_{A^\perp}$), one obtains a double EPW sextic $\widetilde{W}_A$ (resp. $\widetilde{W}_{A^\perp}$) associated to $Z$, coinciding with the ones constructed and studied by O'Grady in \cites{OG1, OG2, OG3}.

    Let $\calM^{GM}$ be the moduli space of GM fourfolds, and let $\calP\colon \calM^{GM}\rightarrow \calD$ be the corresponding period map, assigning $[Z]\mapsto [H^{3,1}(X)],$ where $\calD$ is the associated 20-dimensional quasiprojective period domain (see \cite[Section 4]{Debarre Survey}). Let $\calM^{EPW}$ be the GIT moduli space for EPW sextics, constructed by O'Grady \cite{OG16}. We recall the following result:
    \begin{theorem}\cite[Proposition 6.1]{DK3}\label{thm: period map GM}
        The period map $\calP$ factors as 
        \[\calP\colon\calM^{GM}\twoheadrightarrow \calM^{EPW}\overset{\mathscr{P}}{\longrightarrow}\calD\] where $\mathscr{P}$ is the period map for double EPW sextics, and is an open embedding.

    \end{theorem}

    GM fourfolds that are contained in the same fiber of the period map or in dual fibers are closely related:
    \begin{theorem}\label{thm: bir GM}
        Any two GM fourfolds $Z_1,Z_2\in \calP^{-1}([A])\cup \calP^{-1}([A^\perp])$ are birationally equivalent. Further, $Z_1$ and $Z_2$ are Fourier-Mukai partners.
    \end{theorem}
    \begin{proof}
        The first statement is \cite[Corollary 4.16 and Theorem 4.20]{DK1}. The second statement is \cite[Theorem 1.6]{KuzPerryCatCones}.
    \end{proof}

    For some three-dimensional subspace $V_3\subset V_5$, a plane of the form $\bP(\wedge^2 V_3)$ is called a $\rho$\textit{-plane}.
A GM fourfold $Z$ containing a $\rho$-plane $\Pi$ is birational to a cubic fourfold $X$ containing a cubic scroll $T$: recalling the construction in \cite[Section 7.2]{DIM} and \cite[Proposition 5.6]{KuzPerry}, the complete linear system of quadrics containing the scroll $T$ induces a rational map $q\colon X\dashrightarrow\bP^8$ which is birational onto its image, a GM fourfold $Z_T\subset \Gr(2,V_5)$ containing a $\rho$-plane $\Pi$.  Projection from $\Pi$ induces a rational inverse $f\colon {Z_T}\dashrightarrow X$. As mentioned in \cite[Section 3.3]{BFMQ24}, the proper transform of $Y=\langle T\rangle\cap X$ in $\Bl_TX$ is a small resolution $Y^+$, and $q|_{Y^+}$ is a $\bP^1$-bundle over a plane. Note that for a general cubic fourfold $X\in \calC_{12}$, $X$ contains two families of cubic scrolls $[T], [T^\vee].$ Each choice of scroll gives a non-isomorphic GM-fourfold $Z_T$, $Z_{T^\vee}$; however, the double EPW sextics depend only on the choice of family $[T]$ or $[T^\vee].$ We thus denote them $\widetilde{W}_{T}, \widetilde{W}_{T}^\perp$ and $\widetilde{W}_{T^\vee}, \widetilde{W}_{T^\vee}^\perp$.

     We recall the following result:

    \begin{theorem}\cite[Theorem 1.1]{BFMQ24}\label{thm: EPW for one scroll}
        Let $X\in \calC_{12}$ be a smooth cubic fourfold containing a cubic scroll $T$ such that $\langle T\rangle\cap X$ has 6-nodes in general position. Let $F$ be the Fano variety of lines of $X$. Then $F$, $\widetilde{W}_{T}^\perp$ and $\widetilde{W}_{T^\vee}^\perp$ are all birational.  
    \end{theorem}
The assumption on $\langle T\rangle\cap X$ holds in general and in many special families in $\calC_{12}$.

\section{Cubics containing a non-syzygetic pair of cubic scrolls}\label{sec:nonsyz}

Let $X$ be a very general cubic fourfold containing two cubic scrolls $T_1, T_2\subset X$, such that $[T_1]\cdot [T_2]=1.$ Such a pair of scrolls is called \textit{non-syzygetic} and, together with the square of the hyperplane class, they span a rank $3$ sublattice of $A(X)$. Let $\calC_{nonsyz}\subset \calC_{12}$ be the locus parametrizing such cubic fourfolds. We prove:
\begin{theorem}\label{thm: nonsyz}
    Let $X\in \calC_{nonsyz}$ be a very general cubic fourfold, and $F$ its Fano variety of lines. Then there exists a cubic fourfold $X'\in\calC_{nonsyz}$, i.e. also containing a non-syzgetic pair of cubic scrolls, with Fano variety of lines $F'$ such that:
    \begin{enumerate}
        \item[({\bf{BE}})] $X$ and $X'$ are birational but not isomorphic,
        \item[({\bf{FM}})] $\calA_{X'}\simeq \calA_X,$ and 
        \item[({\bf{BF}})] $F$ and $F'$ are birational but not isomorphic.
    \end{enumerate}
\end{theorem}

The birational geometry of the Fano variety of lines $F$ for such an $X$ was studied in \cite[Theorem 1.3]{BFMQ24}. In particular, we have:
\begin{corollary}\cite[Corollary 1.4, Proposition 5.15]{BFMQ24}\label{cor: bir model for nonsyz}
    There exists a cubic fourfold $X'\in \calC_{nonsyz}$, not isomorphic to $X$, such that there exists a birational morphism $f:F\bir F(X').$ 
\end{corollary}
We will prove that the cubic fourfolds $X$ and $X'$ in \Cref{cor: bir model for nonsyz} are birational and Fourier-Mukai partners.

\begin{proof}[Proof of \Cref{thm: nonsyz}]
    Let $T_1, T_2\subset X$ be a pair of non-syzygetic scrolls contained in $X$. By \Cref{cor: bir model for nonsyz}, $X'$ contains a pair of non-syzygetic cubic scrolls $T_1', T_2'\subset X'.$ From the proof of \cite[Theorem 5.2]{BFMQ24}, the birational map $F\dashrightarrow F(X')$ can be factored as a the composition of two Mukai flops in disjoint Lagrangian planes:
    $$F\dashrightarrow \widetilde{W}_{T_1}^\perp \dashrightarrow F',$$ where $\widetilde{W}_{T_1}^\perp$ is the double EPW sextic described in \Cref{thm: EPW for one scroll} associated to the GM fourfold $Z_{T_1}$. 

    One can construct an inverse by applying the same construction starting at $F'$:
    $$F'\dashrightarrow (\widetilde{W}_{T_1'}')^{\perp}\dashrightarrow F,$$ where $(\widetilde{W}_{T_1'}')^{\perp}$ 
    is associated to a GM fourfold $Z'_{T_1'}$ birational to $X'$. In particular, after possibly relabelling $T_1'$ and $T_1'^\vee$, we can choose $T_1'\subset X'$ such that $\widetilde{W}_{T_1}^\perp\cong (\widetilde{W}_{T_1'}')^\perp$. We prove that $Z'_{T_1'}$ and $Z_{T_1}$ are birational by showing they are contained in the same fiber of the period map $$\calP:\calM^{GM}\rightarrow \calM^{EPW}\rightarrow \calD$$ of \Cref{thm: period map GM}.
    
Let $\calP(Z_{T_1})=A$ and $\calP(Z'_{T_1'})=A'.$ Then, by construction, $\widetilde{W}_{T_1}^\perp = \widetilde{(W)}_{A^\perp}$ and $(\widetilde{W}_{T_1'}')^\perp=\widetilde{W}_{(A')^\perp}$ where we are using the notation of \Cref{subsec: GM4fold}. It follows that $A^\perp=(A')^\perp$, and thus $A=A'.$ By \Cref{thm: bir GM}, we see that $Z_{T_1}, Z'_{T_1'}$ are birational, thus $X\bir X'.$ By combining \cite[Theorem 5.8]{KuzPerry} and \Cref{thm: bir GM}, we see 
\[
\calA_{X}\simeq\calA_{Z_{T_1}}\simeq\calA_{Z_{T_1'}}\simeq\calA_{X'},
\]
where $\calA_{Z_{T_1}}$ is the subcategory of $D^b(Z_{T_1})$ defined in \cite[Proposition 2.3]{KuzPerry}.
\end{proof}

\section{Cubics of discriminant $20$}\label{sec:20}

Fan and Lai proved in \cite{FL24} that a very general cubic fourfold $X\in \calC_{20}$ contains a Veronese surface $V$ and is birational but not isomorphic to another cubic fourfold $X'\in \calC_{20}$ via a Cremona transformation of $\bP^5$ with center $V$. Moreover, $X$ and $X'$ are Fourier-Mukai partners. We prove that the corresponding Fano varieties of lines $F$ and $F'$ are isomorphic. Combining this with the results of \cite{FL24}, we have the following:

\begin{theorem}\label{theorem:20}
Let $X\in \calC_{20}$ be a very general cubic fourfold containing a Veronese surface. Then there exists a cubic fourfold $X'\in \calC_{20}$ with Fano variety of lines $F'$ such that:
    \begin{enumerate}
        \item[({\bf{BE}})] $X$ and $X'$ are birational but not isomorphic,
        \item[({\bf{FM}})] $\calA_X\simeq\calA_{X'}$, and
        \item[({\bf{BF}})] $F\cong F'$.
    \end{enumerate}
\end{theorem}

Further, we show that $X'$ is the only cubic fourfold whose Fano variety of lines is birational to $F$ (\Cref{proposition:20numberofcubics}).

Our outline is as follows: first, in \Cref{subsec:C20lattice} we use lattice-theoretic information about $\NS(F)$ to count the birational models of $F$ up to isomorphism and to count the classes of degree $6$ and divisibility $2$ in $\NS(F)$. Then, in \Cref{subsec:geometry} we use the Cremona transform $X\dashrightarrow X'$ to construct an explicit isomorphism $\pi:F\to F'$.

\subsection{Lattice theory and the movable cone}\label{subsec:C20lattice}
Let $X$ be a very general member of $\calC_{20}$, let $V\subset X$ be the Veronese surface, and let $F$ the Fano variety of lines on $X$. Set $g=\alpha(\eta_X)$, $v=\alpha(V)$, and $\lambda=v-g$, where $\alpha$ denotes the Abel-Jacobi map introduced in Section~\ref{subsec:Fano}. Then under the BBF form, $\NS(F)$ is given by the lattice 
\begin{center}
\begin{tabular}{r|rr}
& $g$ & $\lambda$  \\ \hline
$g$ & $6$   & $2$   \\
$\lambda$ & $2$   & $-6$  \\
\end{tabular}
\end{center}
and $\Orth(\NS(F))$ is generated by $\pm1$ and the matrices 
\[
A=\begin{pmatrix*}[r] 3 & -2 \\ 4 & -3 \end{pmatrix*}\hspace{1cm} B=\begin{pmatrix*}[r] 3 & 4 \\ -2 & -3 \end{pmatrix*}
\]
(which was computed using the algorithm outlined in \cite{Mertens} and implemented using the Magma package \texttt{AutHyp.m}). Both $A$ and $B$ preserve $\mathrm{Pos}(F)$. The discriminant group is
\[
D(\NS(F))=\left\langle\frac{g+2\lambda}{10}\right\rangle\oplus\left\langle\frac{g+\lambda}{4}\right\rangle.
\]

\begin{lemma}\label{lemma:20discgroupaction}
    An isometry $\varphi\in \Orth(\NS(F))$ is induced by a birational automorphism of $F$ if and only if $\varphi$ preserves the positive cone and acts by $\pm\id$ on the subgroup $$H=\bZ/20\bZ\cong\left\langle \frac{g+2\lambda}{5} \right\rangle \oplus \left\langle \frac{g+\lambda}4 \right\rangle.$$
\end{lemma}
\begin{proof}
    Suppose $\varphi=f^*$ for $f\in \Bir(F)$. Then $\varphi$ restricts to a Hodge isometry of $T(F)$, and since $F$ is very general, the only Hodge isometries of $T(F)$ are $\pm\id_{T(F)}.$ By \cite[Prop 3.2.6]{hassett}, the discriminant group $D(T(F))\cong \bZ/20\bZ.$  The overlattice $H^2(F,\bZ)\supset \NS(F)\oplus T(F)$ corresponds to the index two subgroup $H\subset D(NS(F))$ such that $(H, q_{NS(F)})\cong (D(T(F)), -q_{T(F)})$, and hence $\varphi$ acts on $H$ by $\pm \id.$

    Conversely, if $\varphi$ preserves the positive cone and has the prescribed action on $H,$ then one can extend $\varphi$ to an isometry $\tilde{\varphi}\in \Orth(H^2(F,\bZ))$ (here we are using the fact that $q$ does not represent $-2$, so the positive cone and movable cone coincide). By construction, $\tilde{\varphi}$ preserves the Hodge structure, and hence $\varphi \in \im(\pi\colon\Mon^2_{Hdg}(F)\rightarrow \Orth(\NS(F))).$ By \Cref{thm:torelli}, the conclusion follows.
\end{proof}

From this, it is easy to calculate the birational automorphism group:

\begin{lemma}\label{lemma:20inducedbybir}
    The subgroup of $\Orth(\NS(F))$ induced by birational automorphisms is $\langle AB\rangle$. In particular, $\Bir(F)\cong\bZ$.
\end{lemma}
\begin{proof}
    Isometries preserving the positive cone belong to the subgroup $\langle A,B\rangle$. Neither $A$ nor $B$ acts by $\pm\mathrm{Id}$ on the subgroup $H\subset D(\NS(F))$ from Lemma~\ref{lemma:20discgroupaction}, but $AB$ acts on $H$ by $-\mathrm{Id}$. Since $A$ and $B$ are reflections, we deduce that $AB$ generates the subgroup of $\Orth(\NS(F))$ induced by birational automorphisms.

    A very general member of $\calC_{20}$ has no automorphisms by \cite[Theorem 3.8]{BG24}, so the map $\Bir(F)\to\Orth(\NS(F))$ has trivial kernel (for details, see \cite[Proposition 2.18]{BFMQ24}). Since $AB$ has infinite order, we deduce $\Bir(F)\cong\bZ$.
\end{proof}

The orbits of the chambers of $\Mov(F)$ under the action of $\Bir(F)$ correspond to isomorphism classes of birational hyperk\"ahler  models of $F$.

\begin{lemma}\label{lemma:disc20chambers}
    $\Bir(F)$ acts freely and transitively on the chambers of $\Mov(F)$.
\end{lemma}
\begin{proof}
    As previously mentioned, $q$ does not represent $-2$, so $\Mov(F)=\overline{\mathrm{Pos}(F)}$. Let $\rho=ag+b\lambda$ where $6a^2+4ab-6b^2=-10$, so $\rho^\perp$ is a wall in $\Mov(F)$. If $|a|>1$, then $|13a+18b|<|a|$ or $|25a-18b|<|a|$, so exactly one of the classes $AB\cdot\rho$ or $(AB)^{-1}\cdot\rho$ has smaller first coordinate than $\rho$. Inductively, we see that $\rho=\pm(AB)^n\cdot(g+2\lambda)$ for a unique integer $n$. Hence $\Bir(F)$ acts freely and transitively on the walls of $\Mov(F)$. Further, $\Bir(F)$ acts transitively on the chambers of $\Mov(F)$ since $\Mon^2_{Hdg}\cong W_{Bir}\cong \Bir(F)$. 

    It is now easy to see that $\Bir(F)$ acts freely on the chambers of $\Mov(F)$. Indeed, if $(AB)^n$ fixes a chamber of $\Mov(F)$, then $(AB)^{2n}$ fixes both walls of that chamber. Since $\Bir(F)=\langle AB \rangle$ acts freely on the walls, we get $n=0$.
\end{proof}

\begin{lemma}\label{lemma:disc20polarizations}
    $\Nef(F)$ contains exactly two square 6 divisibility $2$ classes, i.e. $g$ and $B\cdot g=3g-2\lambda$.
\end{lemma}
\begin{proof}
    Applying an argument similar to the propagation of solutions to Pell's equation, as in the proof of Lemma~\ref{lemma:disc20chambers}, one finds that any class of square $6$ and divisibility $2$ is of the form 
    \[
    \pm (AB)^n\cdot g \text{ or } \pm (AB)^n\cdot(3g-2\lambda)
    \]
    for a unique integer $n$. The result now follows from the fact that $\Bir(F)$ acts freely on the chambers of $\Mov(F)$.
\end{proof}

\begin{proposition}\label{proposition:20numberofcubics}
    $F$ admits only one birational hyperk\"ahler model up to isomorphism, and there are precisely two smooth cubic fourfolds whose Fano varieties are birational (hence isomorphic) to $F$.
\end{proposition}
\begin{proof}
    The orbits of the action of $\Bir(F)$ on chambers of $\Mov(F)$ correspond to birational models of $F$ up to isomorphism, so Lemma~\ref{lemma:disc20chambers} proves the first claim. Proposition~\ref{prop: existence of cubic} and Lemma~\ref{lemma:disc20polarizations} prove the second claim.
\end{proof}

\subsection{Geometry}\label{subsec:geometry}
In \cite{FL24}, the authors show that a very general $X\in \calC_{20}$ is birational to some non-isomorphic cubic fourfold $X'\in \calC_{20}.$ We briefly review their construction before proving that $X$ and $X'$ have isomorphic Fano varieties of lines.

Let $V\subset \bP^5$ be the Veronese surface, and let $f\colon \bP^5\dashrightarrow \bP^5$ be the birational map given by the linear system $|\calI_V(2)|$. Applying an appropriate projective linear transformation, we can assume that $f$ is an involution. For any cubic fourfold $X$ containing $V$, the image $X':=f(X)$ is again a cubic fourfold containing a Veronese surface $V'$ \cite[Proposition 3.1]{FL24}. We also assume that $X$ is general enough so that $X'$ is smooth.

Let $\Gamma$ be the graph of $f|_X$, which can be identified with \(\Bl_VX\cong\Bl_{V'}X'\), and $p_1\colon\Gamma\to X$ and $p_2\colon\Gamma\to X'$ its projections. Denote by $E$ and $E'$ the classes of the exceptional loci of $p_1$ and $p_2$, respectively, and let $H$ and $H'$ be the pullbacks of the hyperplane classes on $X$ and $X'$. Since $f$ and $f^{-1}$ are induced by the complete linear systems of quadrics containing $V$ and $V'$, respectively, we have
\[
H'=2H-E,\;H=2H'-E'
\]
(cf \cite[Section 3.1]{FL24}). We record some observations:
\begin{itemize}
    \item if $\ell\subset X$ is a line not meeting $V$, then $f(\ell)$ is a conic meeting $V'$ in degree $3$;
    \item if $\ell\subset X$ is a line meeting $V$ in one point $x$, then the proper transform of $\ell$ in $X'$ is a line $\ell'$, whereas the total transform is the union of $\ell'$ and the line $p_2(p_1^{-1}(x))$;
    \item if $\ell\subset X$ is a line meeting $V$ in two points $x$ and $y$ (or one point $x$ with tangency), then the proper transform of $\ell$ is a point, whereas the total transform of $\ell$ is the union of lines $p_2(p_1^{-1}(x))\cup p_2(p_1^{-1}(y))$.
\end{itemize}
In all cases, the total transform of a line $\ell\subset X$ is a conic, which we denote $C_\ell$.

Let $F$ and $F'$ be the Fano varieties of lines on $X$ and $X'$, respectively. We construct a morphism $\pi\colon F\to F'$ as follows: $\pi([\ell])$ is the line residual to $C_\ell$ in the intersection $\langle C_\ell\rangle\cap X'$, where $\langle C_\ell\rangle$ is the plane spanned by $C_\ell$. A morphism $\pi'\colon F'\to F$ is defined analogously.

\begin{proposition}\label{proposition:20isomorphicfanos}
    The maps $\pi\colon F\to F'$ and $\pi'\colon F'\to F$  are mutually inverse isomorphisms.
\end{proposition}
\begin{proof}
    Let $\ell\subset X$ be a line, and let $C_\ell$ be its total transform in $X'$, a conic. By definition, $\pi([\ell])=[\ell']$ where $\ell'$ is the line in $X'$ residual to $C_\ell$. Let $C'_{\ell'}$ be the total transform of $\ell'$ in $X$, again a conic.

    In general, $\ell'$ is disjoint from $V'$, so $f^{-1}$ maps the two points $\ell'\cap C_\ell$ to two points $\ell\cap C'_{\ell'}$. Here, we are using that $f^{-1}\circ f=\mathrm{id}$. In particular, this means $\ell$ and $C'_{\ell'}$ are coplanar, so $\pi'([\ell'])=\ell$. In other words, $\pi'\circ\pi=\mathrm{id}$ on an open set.
\end{proof}

Proposition~\ref{proposition:20numberofcubics} and Proposition~\ref{proposition:20isomorphicfanos} together prove Theorem~\ref{theorem:20}. We finish by identifying the pullback of the Pl\"ucker class under $\pi$, connecting the lattice-theoretic computations of the previous subsection to the geometry outlined in this subsection. Specifically, we prove that the pullback of the Pl\"ucker class by $\pi$ is the unique other square $6$ divisibility $2$ class in $\Nef(F)$, besides $g$.

The two walls of $\Nef(F)$ determined by $(-10)$-classes correspond to Lagrangian planes in $F$. The plane $P_1$ corresponding to the class $g+2\lambda$ is described in \cite[Example 7.16]{ratcurvesholsymp}: it parametrizes the lines in $X$ residual to conics in $V$. We describe the second plane, $P_2$, which necessarily corresponds to the class $11g-8\lambda$. Through a point $x\in V$, there pass exactly two secant lines to $V$ in $X$, in bijection with the two points on $V'$ lying on the line $p_2(p_1^{-1}(x))$. Taking the line in $X$ residual to these secant lines, we obtain a map $V\to F$. A line in $X$ meets at most two secant lines to $V$ since its proper transform can meet $V'$ at most twice, so this map is an embedding, and its image $P_2$ is a plane.

\begin{proposition}
    Let $g$ and $g'$ be the Pl\"ucker classes on $F$ and $F'$, respectively, let $v$ and $v'$ be the images of $[V]$ and $[V']$ under the Abel-Jacobi maps, and let $\lambda=v-g$ and $\lambda'=v'-g'$. Then $\pi^*(g')=3g-2\lambda$.
\end{proposition}
\begin{proof}
    First, we show $\pi(P_1)=P_2'$. Let $[\ell]\in P_1$ and $[\ell']=\pi([\ell])$; then $\ell$ is secant to $V$, so the total transform of $\ell$ under $f$ consists of two secant lines to $V'$, coplanar with $\ell'$. Hence $[\ell']\in P_2'$. Similarly, $\pi(P_2)=P_1'$.

    It follows that $\pi^*(g'+2\lambda')=11g-8\lambda$ and $\pi^*(11g'-8\lambda')=g+2\lambda$, from which we can calculate $\pi^*(g')$.
\end{proof}

\section{Cubics containing a plane and a cubic scroll}\label{sec:C8C12}

Similar to Section~\ref{sec:nonsyz}, we consider another family of cubic fourfolds with associated GM fourfolds: cubics containing a plane and a cubic scroll intersecting in a point. Let $\calC_{M_1}\subset \calC_8\cap \calC_{12}$ be the locus parametrizing cubic fourfolds $X$ containing a plane $P\subset X$ and a cubic scroll $T\subset X$ such that $[P]\cdot [T]=1$. We prove the following:

\begin{theorem}\label{theorem:C8C12}
    Let $X\in \calC_{M_1}$ be a very general cubic fourfold, and let $F$ be the Fano variety of lines on $X$. Then there exists another smooth cubic fourfold $X'\in \calC_{M_1}$ with Fano variety of lines $F'$ such that
    \begin{enumerate}
        \item[({\bf{BE}})] $X$ and $X'$ are birational but not isomorphic, 
        \item[({\bf{FM}})] $\calA_X\simeq\calA_{X'}$, and
        \item[({\bf{BF}})] $F$ and $F'$ are birational but not isomorphic.
    \end{enumerate}
\end{theorem}

We begin in \Cref{subsec:TP1comp} by giving generalities about the three irreducible components of $\calC_8\cap\calC_{12}$, explaining our restriction to the case where $[P]\cdot[T]=1$. We then make geometric observations about $X$ and $F$ in \Cref{subsec:TP1geometry}, including finding a GM fourfold birational to $X$ and containing a pair of disjoint $\rho$-planes. Projecting from one of the planes, we obtain the cubic fourfold $X'$ appearing in Theorem~\ref{theorem:C8C12}. To prove that $X\not\cong X'$, we finish in \Cref{subsec:TP1lattice} with lattice-theoretic calculations and some analysis of $\Mov(F)$.

\subsection{Restriction to $\calC_{M_1}$}\label{subsec:TP1comp}

Suppose $X$ is a smooth cubic fourfold containing a plane $P$ and a cubic scroll $T$, and suppose further that $A(X)=\left\langle \eta_X, P, T\right\rangle$ where $\eta_X$ is the square of the hyperplane class. The intersection form on $A(X)$ takes the form
\begin{center}$M_a:=$
\begin{tabular}{r|rrr}
& $\eta_X$ & $P$ & $T$  \\ \hline
$\eta_X$ & $3$   & $1$  & $3$ \\
$P$ & $1$   & $3$ & $a$ \\
$T$ & $3$ & $a$ & $7$
\end{tabular}
\end{center}
for $a\in\{1,2,3\}$, after possibly relabeling $T$ and $T^\vee$, where $T^\vee$ is a cubic scroll residual to $T$ in a quadric section of $X\cap\langle T\rangle$ \cite[Theorem 1.1]{C8C12}. We denote by $\calC_{M_a}$ the locus of cubic fourfolds admitting a primitive embedding of $M_a\subset A(X)$. We treat only the case $a=1$ here.

\begin{remark}
   When $a=2$, $T$ gives a rational section to the quadric surface fibration $\Bl_PX\to\bP^2$, making $X$ rational.  In the case $a=3$, the hyperplane $H=\langle T\rangle$ contains $P$, so the cubic threefold $Y=X\cap \langle T\rangle$ also contains $P$, and has more than six nodes (see \cite{MV24} and \cite{CMTZ24} for example). Although this case is still interesting (these cubics are conjecturally irrational), Hassett and Tschinkel's results from \cite{flops} do not provide a description of the Fano variety of lines on $Y$. 
\end{remark}

When $a=1$, if $\delta=m[P]+n[T]$, then the sublattice $K_\delta=\left\langle\eta_X,\delta\right\rangle$ has discriminant $4(2m^2+3n^2)$. This shows, \emph{a fortiori}, that $X\not\in\calC_d$ for any $d$ satisfying $(**)$, so $X$ is conjecturally irrational (see \Cref{subsec:special}). Note also that $X\in \calC_{20}$ (cf.~\cite[Theorem~7.11]{yang2021lattice}). This observation is clarified by Lemma~\ref{lemma:C8C12directrix}, where we show that $X$ contains a cubic scroll intersecting $P$ along its directrix: as discussed in \cite[Proposition 3.4]{C8C12}, the union of these two surfaces is a flat limit of the Veronese surface.

\begin{remark}\label{remark:C8C12degeneration}
    Although $X\in\calC_{20}$, Theorem~\ref{theorem:C8C12} cannot be deduced from Theorem~\ref{theorem:20} or from the results from \cite{FL24}, since $X$ is no longer very general in $\calC_{20}$. Indeed, the situation here differs: unlike for the very general member of $\calC_{20}$, $F$ and $F'$ are merely birational, not isomorphic.
\end{remark}

\subsection{Geometry}\label{subsec:TP1geometry}

Proceeding with the case $[P]\cdot[T]=1$, we record some geometric observations. Let $H=\langle T \rangle$ and $Y=X \cap H$. The example from \cite[Appendix A]{C8C12} shows that, in general, $Y$ has six nodes in general linear position. We will assume this generality condition in what follows. 

Bezout's theorem guarantees that $P$ intersects $H$ along a line $L$ (rather than being contained in $H$).
\begin{lemma}\label{lemma:C8C12directrix}
    There is a unique cubic scroll $T_L$ homologous to $T$ and a unique cubic scroll $T_L^\vee$ homologous to $T^\vee$ such that $L\subset T_L$ and $L\subset T_L^\vee$. Moreover, $L$ is a section of both cubic scrolls.
\end{lemma}
\begin{proof}
     First, we note that $P$ passes through none of the nodes of $Y$. Indeed, if $y$ is a node of $Y$, and $y\in P$, then we would have $P\subset T_yY=H$. 
     
     Since $L$ passes through none of the nodes of $Y$, $[L]$ is a smooth point of $F(Y)$. In particular, $[L]$ belongs to at most one irreducible component. By \cite[Proposition 4.1]{flops}, there is a decomposition
    \[
    F(Y)=\bP(V)\cup S'\cup\bP(V^\vee),
    \]
    where a point $[\ell]\in \bP(V)$ (respectively $\bP(V^\vee)$) parametrizes a line in the ruling of some cubic scroll homologous to $T$ (respectively $T^\vee$), and a point $[\ell]\in S'$ parametrizes a line that is the directrix of exactly two cubic scrolls, one homologous to $T$ and the other homologous to $T^\vee$. We will show $[L]\in S'$ by checking that $[P]\cdot[T]=0$ if $L$ is a fiber of a scroll homologous to $T$.
     Indeed, in that case, we would have 
    \[
    [P]\cdot[T]=[P]\cdot[T']=-\deg(N_{L/T'})=0,
    \]
    following from the calculation in \cite[Section 2.4]{C8C12}.
\end{proof}

As a consequence of the above, we find disjoint pairs of planes in the Fano variety $F$ of lines on $X$.

\begin{corollary}\label{corollary:C8C12disjointplanes}
    The plane $P^*\subset F$ dual to $P$ is disjoint from $\bP(V)$ and $\bP(V^\vee)$.
\end{corollary}
\begin{proof}
    The only way $P^*$ could intersect either $\bP(V)$ or $\bP(V^\vee)$ is if $L$ is a line in the ruling of some cubic scroll in $X$. By Lemma~\ref{lemma:C8C12directrix}, this is not the case.
\end{proof}

We also find a pair of $\rho$-planes in two GM fourfolds associated to $X$.

\begin{proposition}\label{proposition:TP1 two planes}
    Let $T_L$ and $T_L^\vee$ be the cubic scrolls in $X$ whose directrices lie on $P$. Let $Z_{T_L}$ and $Z_{T_L^\vee}$ be the GM fourfolds associated to the pairs $(X,T_L)$ and $(X,T_L^\vee)$, respectively. Then $Z_{T_L}$ and $Z_{T^\vee_L}$ each contain a pair of disjoint $\rho$-planes.
\end{proposition}
\begin{proof}
    Recall from \Cref{subsec: GM4fold} that the complete linear system of quadrics containing $T_L$ induces a map $q\colon\Bl_{T_L}X\to Z_{T_L}$. The proper transform of $Y$ is a plane $\Pi\subset Z_{T_L}$, and projection from $\Pi$ gives a rational inverse to $q$. Since $P$ intersects $T_L$ along a line, the proper transform of $P$ in $\Bl_{T_L}X$ is again a plane. Away from $Y$, the rational map $q\colon X\dashrightarrow Z_{T_L}$ is an open embedding, so $q|_P$ is a birational equivalence. Since $P$ is a minimal surface, $q(P)$ is a plane, which we call $\Pi'$.

    We now argue that $\Pi$ and $\Pi'$ are disjoint. Indeed, projection from $\Pi$ yields a rational inverse to $q$. The only way for $\Pi'$ to project to a plane is if $\Pi'\subset\Pi^\perp$, i.e. $\Pi\cap\Pi'=\emptyset$.
    Note that we have a Hodge-isometry between transcendental lattices $T(X)\cong T(Z_{T_L})$: in particular, $\textrm{disc}(T(Z_{T_L}))=32.$ If $\Pi'$ is not a $\rho$-plane, then by \cite[Proposition 6.2]{DIM} we have $\textrm{disc}(T(X_{T_L}))=29,$ a contradiction. 

    The argument for $Z_{T_L^\vee}$ is identical.
\end{proof}

\begin{definition}\label{defi:X'}
    We define $X'$ to be the cubic fourfold obtained as follows: let $\Pi'$ be the proper transform of $P$ in $Z_{T_L}$, which is a $\rho$-plane. Projecting $Z_{T_L}$ from $\Pi'$ yields a morphism $\Bl_{\Pi'}Z_{T_L}\to X'$ whose image is another cubic fourfold. 
\end{definition}

Moreover, $X'$ contains a plane $P'$ (the image of $\Pi$ under the linear projection) and a cubic scroll $T'$, such that the complete linear system of quadrics containing $T'$ induces a map $\Bl_{T'}X'\to Z_{T_L}$, and $\Pi$ is the image of the hyperplane section $X'\cap\langle T'\rangle$.

\begin{remark}
    The rational map $X\dashrightarrow X'$ can also be realized via the complete linear system of quadrics containing $T_L\cup P$, which is a degeneration of the Veronese surface, as mention in Remark~\ref{remark:C8C12degeneration}. Hence our construction matches the Cremona transformation from \cite{FL24}. However, we rely on the intermediate maps $X\dashrightarrow Z_{T_L}$ and $Z_{T_L}\dashrightarrow X'$ to obtain a derived equivalence between $X$ and $X'$.
\end{remark}

\begin{proposition}\label{proposition:TP1 relating cubics}
    Let $X$ and $X'$ be as in \Cref{defi:X'}, and $F$ and $F'$ their Fano varieties of lines. Then $F'$ is isomorphic to the Mukai flop of $F$ along the disjoint pair of Lagrangian planes $P^*$ and $\bP(V^\vee)$.
\end{proposition}
\begin{proof}
 If $[\ell]\not\in \bP(V^\vee)\cup P^*$, then
    \begin{itemize}
        \item the total transform of $\ell$  under the birational map $X\dashrightarrow Z_{T_L}$ is a conic $C_\ell$ meeting $\Pi$ in a point (here we use that $[\ell]\not\in\bP(V^\vee)$);
        \item since $[\ell]\not\in P^*$, $C_\ell\not\subset\Pi'$, so the total transform of $C_\ell$ under the projection $Z_{T_L}\dashrightarrow X'$ is a conic $D_\ell$.
    \end{itemize}
        Let $\ell'$ be the line residual to $D_\ell$ in $X'$. Moreover, let $C_{\ell'}$ be the total transform of $\ell'$ under the birational map $X'\dashrightarrow Z_{T_L}$, and let $Q_\ell$ and $Q_{\ell'}$ be the singular quadric threefolds associated to the conics $C_\ell$ and $C_{\ell'}$ as in \cite[Lemma 2.22]{BFMQ24}.
        
        We first show that the formula $\pi([\ell])=[\ell']$ defines the same birational equivalence $\pi:F\dashrightarrow F'$ as the composition of birational equivalences 
        \[
        F\dashrightarrow\widetilde{W}_{T_L}^\perp\overset{\tau}\to\widetilde{W}_{T_L}^\perp\dashrightarrow F'
        \]
        defined in \cite[Proposition 3.6]{BFMQ24}, where $\tau$ is the covering involution on $\widetilde{W}_{T_L}^\perp$. This amounts to checking that $Q_\ell=Q_{\ell'}$ and that $C_\ell$ and $C_{\ell'}$ span planes in opposite rulings of this quadric surface. Indeed, conics in the same ruling of a singular quadric threefold, or conics in two distinct singular quadric threefolds parametrized by $\widetilde{W}_{T_L}^\perp$ meet in at most one point, whereas $C_\ell$ and $C_{\ell'}$ meet in two points, corresponding to the intersection of $D_\ell$ and $\ell'$.
        
        Already, we defined $\pi$ away from $\bP(V^\vee)\cup P^*$. To see that $\pi$ is indeterminate along $\bP(V^\vee)\cup P^*$, it suffices to check that the birational maps $f\colon\widetilde{W}_{T_L}^\perp\dashrightarrow F$ and $f'\colon\widetilde{W}_{T_L}^\perp\overset{\tau}\to\widetilde{W}_{T_L}^\perp\dashrightarrow F'$ have different base loci. Following~\cite[Lemma 3.5 and Proposition 3.6]{BFMQ24}, a point $w\in\widetilde{W}_{T_L}^\perp$ parametrizes a singular quadric threefold $Q_w$, and $f$ (respectively $f'$) is defined at $w$ if and only if $Q_w\cap\Pi=\emptyset$ (respectively $Q_w\cap \Pi'=\emptyset$). A singular quadric threefold cannot contain a disjoint pair of planes, so the base loci of $f$ and $f'$ disagree, as needed.
\end{proof}

\begin{remark}
    Note that one can also obtain another cubic fourfold $X''$ by projecting $Z_{T_L^\vee}$ from the proper transform of $P$. Although we do not include the proof here, $X'$ and $X''$ are isomorphic.
\end{remark}

\subsection{Lattice theory and the movable cone}\label{subsec:TP1lattice}

Let $F$ be the Fano variety of lines on $X$. The BBF form on $\NS(F)$ is given by: 
\begin{center}
\begin{tabular}{r|rrr}
& $g$ & $p$ & $\lambda$  \\ \hline
$g$ & $6$   & $2$  & $0$ \\
$p$ & $2$   & $-2$ & $0$ \\
$\lambda$ & $0$ & $0$ & $-4$
\end{tabular}
\end{center}
where $g$, $p$, and $\lambda$ are the images under the Abel-Jacobi map of $\eta_X$, $[P]$, and $[T]-\eta_X$, respectively. The isometry group $\Orth(\NS(F))$ is generated by
$\pm1$ and the reflections
\[G_1=\begin{pmatrix} 5 & 2 & -4 \\ 0 & -1 & 0 \\ 6 & 2 & -5 \end{pmatrix},\hspace{0.2cm}
G_2=\begin{pmatrix} 3 & 2 & -2 \\ -2 & -1 & 2 \\ 2 & 2 & -1 \end{pmatrix},
\]
\[G_3=\begin{pmatrix}  1 & 0 & 0 \\ 0 & 1 & 0 \\ 0 & 0 & -1 \end{pmatrix},\hspace{0.2cm}
 G_4=\begin{pmatrix} 1 & 0 & 0 \\ 2 & -1 & 0 \\ 0 & 0 & 1 \end{pmatrix}.\]

The discriminant group is:
\[
D(\NS(F))=\left\langle\frac g2\right\rangle \oplus \left\langle \frac{3g-p}8 \right\rangle \oplus \left\langle \frac\lambda4\right\rangle.
\]

\begin{lemma}\label{lemma:C8C12discgroup}
    An isometry of $\NS(F)$ is a monodromy operator if and only if it acts by $\pm\id$ on the subgroup 
    \[
    H= \left\langle \frac{3g-p}8 \right\rangle \oplus \left\langle \frac\lambda4\right\rangle
    \]
    In particular, an isometry is induced by a birational automorphism if and only if it preserves $\Mov(F)$ and acts by $\pm\id$ on $H$.
\end{lemma}
\begin{proof}
    The argument is the same as in Lemma~\ref{lemma:20discgroupaction}.
\end{proof}

Let $F_1^\vee$ be the flop of $F$ along $\bP(V^\vee)$, and let $\iota^\vee$ be the rational involution on $F$ defined in \cite[Theorem 6.2]{flops} that becomes regular on $F_1^\vee$ (see \cite[Proposition 3.7]{BFMQ24}).

\begin{lemma}\label{lemma:C8C12iotavee}
     The action of $(\iota^\vee)^*$ on $\NS(F)$ is given by $G_1$.
\end{lemma}
\begin{proof}
    By \cite[Proposition~6.5]{flops}, $(\iota^\vee)^*(g+\lambda)=g+\lambda$. Direct calculation shows that the only isometry of $\NS(F)$ fixing $g+\lambda$ and acting by $\pm\id$ on the subgroup $H$ of $D(\NS(F))$ from Lemma~\ref{lemma:C8C12discgroup} is $G_1$.
\end{proof}

We now work toward understanding the pullback of the Pl\"ucker class on $F'$ via the rational map $\pi\colon F\dashrightarrow F'$.

\begin{lemma}\label{lemma:C8C12sq6}
    The only class of degree $6$ in $\Nef(F)$ is $g$.
\end{lemma}
\begin{proof}
    As in \cite[Proposition 7.2]{flops}, the planes $\bP(V)$ and $\bP(V^\vee)$ in $F$ yield walls $(g\pm2\lambda)^\perp$ in $\Mov(F)$ on the boundary of $\Nef(F)$. The plane $P^*\subset F$ dual to $P$ yields a wall $(g-2p)^\perp$ also on the boundary of $\Nef(F)$. Further, the variety of lines incident to $P$ composes a conic bundle over a K3 surface, and the wall $p^\perp$ lies on the boundary of both $\Nef(F)$ and $\Mov(F)$. Hence,\[
    \Nef(F)\subset\Cone((g+2\lambda)^\perp,(g-2\lambda)^\perp,(g-2p)^\perp,p^\perp).
    \]
    Direct computation shows that $g$ is the unique class of degree $6$ in this cone pairing positively with all of $g\pm2\lambda$, $g-2p$, and $p$.
\end{proof}

\begin{lemma}\label{lemma:C8C12plucker}
    Let $\pi\colon F\dashrightarrow F'$ be the birational equivalence described in the proof of Proposition~\ref{proposition:TP1 relating cubics}, and let $g'$ be the Pl\"ucker class on $F'$. Then $\pi^*(g')=3g-2p+2\lambda$. 
\end{lemma}
\begin{proof}
    Note that $(g-2p)^\perp\cap(g+2\lambda)^\perp$ is spanned by $\nu=3g-p+2\lambda$. These two walls correspond to Lagrangian planes $P,\bP(V^\vee)\subset F$ that are disjoint, by Lemma~\ref{corollary:C8C12disjointplanes}. In particular, the two planes can be flopped simultaneously on $F$, so $\nu$ lies on the boundary of $\Nef(F)$. Also, $\nu$ lies on the boundary of the chambers corresponding to the flop of $F$ along $\bP(V^\vee)$ (i.e. $\Nef(F_1^\vee$) and the flop of $F$ along $\bP(V^\vee)\cup P^*$ (i.e. $\pi^*\Nef(F')$, by Proposition~\ref{proposition:TP1 relating cubics}).
    
    It is straightforward to calculate that $g+2\lambda$ and $g-2p$ are the only wall divisors orthogonal to $\nu$. Hence there are four chambers of $\Mov(F)$ on whose boundary $\nu$ lies. Since $G_2\cdot\nu=\nu$, $G_2$ permutes these four chambers.

    We claim that $G_2\cdot\Nef(F_1^\vee)=\Nef(F_1^\vee)$. Since $\iota^\vee$ is regular on $F_1^\vee$, and $(\iota^\vee)^*(g+\lambda)=g+\lambda$ (verified using Lemma~\ref{lemma:C8C12iotavee}), we know $g+\lambda\in\Nef(F_1^\vee)$. By calculating $G_2(g+\lambda)=g+\lambda$, we verify that $G_2$ sends $\Nef(F_1^\vee)$ to itself.

    This forces $G_2\cdot\Nef(F)=\Nef(F')$. In particular, $G_2\cdot g=3g-2p+2\lambda$ is the unique class of degree $6$ in $\Nef(F')$ by Lemma~\ref{lemma:C8C12sq6}. This must then be the pullback of the Pl\"ucker class on $F'$.
\end{proof}

We can now apply \Cref{prop: existence of cubic}:

\begin{proposition}\label{proposition:C8C12nonisomorphic}
    The Fanos $F$ and $F'$ are not isomorphic, so in particular $X\not\cong X'$.
\end{proposition}
\begin{proof}
    The fact that the second claim follows from the first is obvious. By Proposition~\ref{prop: existence of cubic}, Lemma~\ref{lemma:C8C12discgroup}, and Lemma~\ref{lemma:C8C12plucker}, it suffices to check that there is no isometry $\phi\in\Orth(\NS(F))$ such that $\phi(g)=3g-2p+2\lambda$ and acting by $\pm\id$ on 
    \[
    H=\left\langle\frac{3g-p}{8}\right\rangle\oplus\left\langle\frac{\lambda}{4}\right\rangle\subset D(\NS(F)).
    \]
    Indeed, the only isometries satisfying $\phi(g)=3g-2p+2\lambda$ are $G_2$ and $G_2G_3$, and neither acts on $H$ by $\pm\id$.
\end{proof}

We now prove our main theorem:

\begin{proof}[Proof of \Cref{theorem:C8C12}]
 For ({\bf{BE}}), note that $X$ and $X'$ are both birational to the GM fourfold $Z_{T_L}$. Moreover, since $Z_{T_L}$ is the GM fourfold associated to the pairs $(X,T_L)$ and $(X',T')$, \cite[Theorem 5.8]{KuzPerry} yields
    \[
    \calA_X\simeq\calA_{Z_{T_L}}\simeq\calA_{X'},
    \]
proving ({\bf{FM}}). In \Cref{proposition:TP1 relating cubics}, we proved ({\bf{BF}}), but in \Cref{proposition:C8C12nonisomorphic}, we prove $F\not\cong F'$, which allows us to conclude that $X\not\cong X'$.
\end{proof}

\section{Cubics of discriminant $546$}\label{sec:546}

To see that neither Fourier-Mukai partnership nor a derived equivalence between the Fano varieties implies that two cubic fourfolds have birationally equivalent Fano varieties of lines, turn to discriminant $546$. The example is inspired by \cite{Meachanetal}.

\begin{theorem}\label{theorem:546}
    Let \(X\in \calC_{546}\) be a very general cubic fourfold and $F$ its Fano variety of lines. Then there exists a cubic fourfold \(X'\in \calC_{546}\) with Fano variety of lines $F'$ such that: 
    \begin{enumerate}
        \item[({\bf{BE}})] $X$ and $X'$ are conjecturally birational,
        \item[({\bf{FM}})] \(\calA_X\simeq \calA_{X'}\),
        \item[$\neg$({\bf{BF}})] \(F\) and \(F'\) are not birational, and
        \item[({\bf{DF}})] \(D^b(F)\simeq D^b(F')\).
    \end{enumerate}
\end{theorem}

Following \cite{Brakkee}, the key ingredient of this result is the fact that the Fano variety $F$ of lines on a cubic fourfold $X$ of discriminant $d$ satisfying condition condition \((***)\) is birational to the Hilbert square of points on a K3 surface, as mentioned in \Cref{subsec:special}. This also implies that $X$ is conjecturally rational.

We start with lattice-theoretic observations about $\NS(F)$, whose BBF form is given below.
\begin{center}
\begin{tabular}{r|rr}
& $g$  & $\lambda$  \\ \hline
$g$ & $6$   & $0$   \\
$\lambda$ & $0$ & $-182$ 
\end{tabular}
\end{center}

\begin{lemma}\label{lemma:546square6classes}
    The only square $6$ class in $\Mov(F)$ is $g$.
\end{lemma}
\begin{proof}
    The classes $\rho_1=11g-2\lambda$ and $\rho_2=11g+2\lambda$ both square to $-2$ and pair positively with $g$, so any class in $\Mov(F)$ pairs nonnegatively with $\rho_1$ and $\rho_2$. Hence square $6$ classes $ag+b\lambda$ in $\Mov(F)$ correspond to integer solutions to the system of equations
    \[
    6a^2-182b^2=6,\; 66a-364b>0,\; 66a+364b>0.
    \]
    The only solution is $a=1$, $b=0$.
\end{proof}

\begin{proof}[Proof of Theorem~\ref{theorem:546}]
By \cite[Proposition 2.6]{FL24}, \(X\) has two Fourier-Mukai partners; let \(X'\) be the nontrivial one, so that \(X\not\cong X'\). By \cite[Proposition~3.4]{Huy17}, the equivalence \(\calA_X\simeq \calA_{X'}\) gives a Hodge isometry \(\widetilde{H}(X,\bZ)\cong \widetilde{H}(X',\bZ)\), restricting to a Hodge isometry \(T(X)\cong T(X')\), thus \(X'\in \calC_{546}\). Since \(d=546\) satisfies condition $(***)$ (namely, with \(a=1, n=16\)), it follows by \cite[Theorem~2]{AddingtonTwoRatConjs} that \(F\) is birational to \(\Hilb^2(S)\) for a K3 surface \(S\) satisfying \(D^b(S)\simeq \calA_X\). In fact, more is true: by \cite[Prop.~4.5, Remark~4.6]{Brakkee}, since \(3 \mid d\), we have \(F\cong \Hilb^2(S)\) since \(\Hilb^2(S)\) has only one birational model. Similarly we have \(F'\cong \Hilb^2(S')\) for a K3 surface \(S'\) satisfying \(D^b(S')\simeq \calA_{X'}\). 

We now argue that $F$ and $F'$ are not birational. Since $X\not\cong X'$, a birational equivalence $F\bir F'$ would imply the existence of (at least) two orbits of square $6$ divisibility $2$ classes in $\Mov(F)$ under the action of $\Bir(F)$, by Proposition~\ref{prop: existence of cubic}. This is not the case, according to Lemma~\ref{lemma:546square6classes}.

For the final claim, we note that 
\[D^b(S)\simeq \mathcal{A}_X\simeq \mathcal{A}_{X'}\simeq D^b(S'),\] 
and thus 
\[D^b(F)\simeq D^b(\mathrm{Hilb}^2(S))\simeq D^b(\mathrm{Hilb}^2(S'))\simeq D^b(F'),\] 
by \cite[Proposition~8]{Ploog}.

\end{proof}

\begin{remark}\label{remark:546numerics}
It is likely that other examples of cubic fourfolds of discriminant \(d\) can be constructed that satisfy Theorem~\ref{theorem:546}. When \(d\) satisfies \((***)\) with \(d =6pq\) where \(p, q >3\) are distinct primes, such a cubic fourfold has a non-trivial Fourier-Mukai partner, and the results of \cite{Brakkee} guarantee that the Fano variety of lines has a unique birational model. An understanding of how the birational automorphism group acts on line bundles of degree $6$ and divisibility $2$ (i.e. a statement analogous to Lemma~\ref{lemma:546square6classes}) would be necessary for any fixed \(d\) to complete the argument. Such cubic fourfolds are conjecturally rational.
\end{remark}

\bibliographystyle{alpha}
\bibliography{bibliography}

% \bib, bibdiv, biblist are defined by the amsrefs package.
\begin{bibdiv}
\begin{biblist}

\bib{AddingtonTwoRatConjs}{article}{
      author={Addington, Nicolas},
       title={On two rationality conjectures for cubic fourfolds},
        date={2016},
        ISSN={1073-2780,1945-001X},
     journal={Math. Res. Lett.},
      volume={23},
      number={1},
       pages={1\ndash 13},
         url={https://doi.org/10.4310/MRL.2016.v23.n1.a1},
      review={\MR{3512874}},
}

\bib{addington}{article}{
      author={Addington, Nicolas},
      author={Thomas, Richard},
       title={Hodge theory and derived categories of cubic fourfolds},
        date={2014},
        ISSN={0012-7094},
     journal={Duke Math. J.},
      volume={163},
      number={10},
       pages={1885\ndash 1927},
         url={https://doi.org/10.1215/00127094-2738639},
      review={\MR{3229044}},
}

\bib{C8C12}{misc}{
      author={Bolognesi, Michele},
      author={Brahimi, Zakaria},
      author={Awada, Hanine},
       title={Moduli of cubic fourfolds and reducible {OADP} surfaces},
        date={2024},
         url={https://arxiv.org/abs/2409.12032},
}

\bib{Magma}{article}{
      author={Bosma, Wieb},
      author={Cannon, John},
      author={Playoust, Catherine},
       title={The {M}agma algebra system. {I}. {T}he user language},
        date={1997},
        ISSN={0747-7171},
     journal={J. Symbolic Comput.},
      volume={24},
      number={3-4},
       pages={235\ndash 265},
         url={http://dx.doi.org/10.1006/jsco.1996.0125},
        note={Computational algebra and number theory (London, 1993)},
      review={\MR{MR1484478}},
}

\bib{BD85}{article}{
      author={Beauville, Arnaud},
      author={Donagi, Ron},
       title={La vari\'{e}t\'{e} des droites d'une hypersurface cubique de
  dimension {$4$}},
        date={1985},
        ISSN={0249-6291},
     journal={C. R. Acad. Sci. Paris S\'{e}r. I Math.},
      volume={301},
      number={14},
       pages={703\ndash 706},
      review={\MR{818549}},
}

\bib{Beckmann}{article}{
      author={Beckmann, Thorsten},
       title={Derived categories of hyper-{K}\"ahler manifolds: extended
  {M}ukai vector and integral structure},
        date={2023},
        ISSN={0010-437X,1570-5846},
     journal={Compos. Math.},
      volume={159},
      number={1},
       pages={109\ndash 152},
         url={https://doi.org/10.1112/S0010437X22007849},
      review={\MR{4543656}},
}

\bib{BFMQ24}{misc}{
      author={Brooke, Corey},
      author={Frei, Sarah},
      author={Marquand, Lisa},
      author={Qin, Xuqiang},
       title={Birational geometry of fano varieties of lines on cubic fourfolds
  containing pairs of cubic scrolls},
        date={2024},
         url={https://arxiv.org/abs/2407.18904},
}

\bib{BG24}{misc}{
      author={Billi, Simone},
      author={Grossi, Annalisa},
       title={Non-symplectic automorphisms of prime order of {O}'{G}rady's
  tenfolds and cubic fourfolds},
        date={2024},
}

\bib{BHT}{article}{
      author={Bayer, Arend},
      author={Hassett, Brendan},
      author={Tschinkel, Yuri},
       title={Mori cones of holomorphic symplectic varieties of {K}3 type},
        date={2015},
        ISSN={0012-9593,1873-2151},
     journal={Ann. Sci. \'{E}c. Norm. Sup\'{e}r. (4)},
      volume={48},
      number={4},
       pages={941\ndash 950},
         url={https://doi.org/10.24033/asens.2262},
      review={\MR{3377069}},
}

\bib{MR4292740}{article}{
      author={Bayer, Arend},
      author={Lahoz, Mart\'{\i}},
      author={Macr\`\i, Emanuele},
      author={Nuer, Howard},
      author={Perry, Alexander},
      author={Stellari, Paolo},
       title={Stability conditions in families},
        date={2021},
        ISSN={0073-8301},
     journal={Publ. Math. Inst. Hautes \'{E}tudes Sci.},
      volume={133},
       pages={157\ndash 325},
         url={https://doi.org/10.1007/s10240-021-00124-6},
      review={\MR{4292740}},
}

\bib{Brakkee}{article}{
      author={Brakkee, Emma},
       title={Two polarised {K}3 surfaces associated to the same cubic
  fourfold},
        date={2021},
        ISSN={0305-0041,1469-8064},
     journal={Math. Proc. Cambridge Philos. Soc.},
      volume={171},
      number={1},
       pages={51\ndash 64},
         url={https://doi.org/10.1017/S0305004120000055},
      review={\MR{4268803}},
}

\bib{BolognesiRusso}{article}{
      author={Bolognesi, Michele},
      author={Russo, Francesco},
      author={Staglianò, Giovanni},
       title={Some loci of rational cubic fourfolds},
        date={2015},
      eprint={1504.05863v1},
}

\bib{CMTZ24}{misc}{
      author={Cheltsov, Ivan},
      author={Marquand, Lisa},
      author={Tschinkel, Yuri},
      author={Zhang, Zhijia},
       title={Equivariant geometry of singular cubic threefolds, ii},
        date={2024},
         url={https://arxiv.org/abs/2405.02744},
}

\bib{DebarreSurvey}{misc}{
      author={Debarre, Olivier},
       title={Gushel-mukai varieties},
        date={2020},
         url={https://arxiv.org/abs/2001.03485},
}

\bib{debarre2020hyperkahler}{article}{
      author={Debarre, Olivier},
       title={Hyperk\"ahler manifolds},
        date={2020},
      eprint={1810.02087},
}

\bib{DIM}{incollection}{
      author={Debarre, O.},
      author={Iliev, A.},
      author={Manivel, L.},
       title={Special prime {F}ano fourfolds of degree 10 and index 2},
        date={2015},
   booktitle={Recent advances in algebraic geometry},
      series={London Math. Soc. Lecture Note Ser.},
      volume={417},
   publisher={Cambridge Univ. Press, Cambridge},
       pages={123\ndash 155},
      review={\MR{3380447}},
}

\bib{DK1}{article}{
      author={Debarre, Olivier},
      author={Kuznetsov, Alexander},
       title={Gushel-{M}ukai varieties: classification and birationalities},
        date={2018},
        ISSN={2313-1691,2214-2584},
     journal={Algebr. Geom.},
      volume={5},
      number={1},
       pages={15\ndash 76},
         url={https://doi.org/10.14231/ag-2018-002},
      review={\MR{3734109}},
}

\bib{DK3}{article}{
      author={Debarre, Olivier},
      author={Kuznetsov, Alexander},
       title={Gushel-{M}ukai varieties: moduli},
        date={2020},
        ISSN={0129-167X,1793-6519},
     journal={Internat. J. Math.},
      volume={31},
      number={2},
       pages={2050013, 59},
         url={https://doi.org/10.1142/S0129167X20500135},
      review={\MR{4083237}},
}

\bib{FL24}{article}{
      author={Fan, Yu-Wei},
      author={Lai, Kuan-Wen},
       title={New rational cubic fourfolds arising from {C}remona
  transformations},
        date={2023},
        ISSN={2313-1691,2214-2584},
     journal={Algebr. Geom.},
      volume={10},
      number={4},
       pages={432\ndash 460},
         url={https://doi.org/10.14231/ag-2023-014},
      review={\MR{4606409}},
}

\bib{galkinshinder}{misc}{
      author={Galkin, Sergey},
      author={Shinder, Evgeny},
       title={The {F}ano variety of lines and rationality problem for a cubic
  hypersurface},
        date={2014},
         url={https://arxiv.org/pdf/1405.5154},
}

\bib{hassett}{article}{
      author={Hassett, Brendan},
       title={Special cubic fourfolds},
        date={2000},
        ISSN={0010-437X},
     journal={Compositio Math.},
      volume={120},
      number={1},
       pages={1\ndash 23},
         url={https://doi.org/10.1023/A:1001706324425},
      review={\MR{1738215}},
}

\bib{hassett-thesis}{book}{
      author={Hassett, Brendan~Edward},
       title={Special cubic hypersurfaces of dimension four},
   publisher={ProQuest LLC, Ann Arbor, MI},
        date={1996},
  url={http://gateway.proquest.com.proxy.library.stonybrook.edu/openurl?url_ver=Z39.88-2004&rft_val_fmt=info:ofi/fmt:kev:mtx:dissertation&res_dat=xri:pqdiss&rft_dat=xri:pqdiss:9631502},
        note={Thesis (Ph.D.)--Harvard University},
      review={\MR{2694332}},
}

\bib{ratcurvesholsymp}{article}{
      author={Hassett, B.},
      author={Tschinkel, Y.},
       title={Rational curves on holomorphic symplectic fourfolds},
        date={2001},
        ISSN={1016-443X,1420-8970},
     journal={Geom. Funct. Anal.},
      volume={11},
      number={6},
       pages={1201\ndash 1228},
         url={https://doi.org/10.1007/s00039-001-8229-1},
      review={\MR{1878319}},
}

\bib{flops}{article}{
      author={Hassett, Brendan},
      author={Tschinkel, Yuri},
       title={Flops on holomorphic symplectic fourfolds and determinantal cubic
  hypersurfaces},
        date={2010},
        ISSN={1474-7480},
     journal={J. Inst. Math. Jussieu},
      volume={9},
      number={1},
       pages={125\ndash 153},
         url={https://doi.org/10.1017/S1474748009000140},
      review={\MR{2576800}},
}

\bib{Huy17}{article}{
      author={Huybrechts, Daniel},
       title={The {K}3 category of a cubic fourfold},
        date={2017},
        ISSN={0010-437X,1570-5846},
     journal={Compos. Math.},
      volume={153},
      number={3},
       pages={586\ndash 620},
         url={https://doi.org/10.1112/S0010437X16008137},
      review={\MR{3705236}},
}

\bib{Huybrechtscubicsbook}{book}{
      author={Huybrechts, Daniel},
       title={The geometry of cubic hypersurfaces},
      series={Cambridge Studies in Advanced Mathematics},
   publisher={Cambridge University Press, Cambridge},
        date={2023},
      volume={206},
        ISBN={978-1-009-28000-6; [9781009280020]},
      review={\MR{4589520}},
}

\bib{kawamata}{article}{
      author={Kawamata, Yujiro},
       title={{$D$}-equivalence and {$K$}-equivalence},
        date={2002},
        ISSN={0022-040X,1945-743X},
     journal={J. Differential Geom.},
      volume={61},
      number={1},
       pages={147\ndash 171},
         url={http://projecteuclid.org/euclid.jdg/1090351323},
      review={\MR{1949787}},
}

\bib{KuzPerry}{article}{
      author={Kuznetsov, Alexander},
      author={Perry, Alexander},
       title={Derived categories of {G}ushel-{M}ukai varieties},
        date={2018},
        ISSN={0010-437X,1570-5846},
     journal={Compos. Math.},
      volume={154},
      number={7},
       pages={1362\ndash 1406},
         url={https://doi.org/10.1112/s0010437x18007091},
      review={\MR{3826460}},
}

\bib{KuzPerryCatCones}{article}{
      author={Kuznetsov, Alexander},
      author={Perry, Alexander},
       title={Categorical cones and quadratic homological projective duality},
        date={2023},
        ISSN={0012-9593,1873-2151},
     journal={Ann. Sci. \'Ec. Norm. Sup\'er. (4)},
      volume={56},
      number={1},
       pages={1\ndash 57},
      review={\MR{4563868}},
}

\bib{kuzcubic}{incollection}{
      author={Kuznetsov, Alexander},
       title={Derived categories of cubic fourfolds},
        date={2010},
   booktitle={Cohomological and geometric approaches to rationality problems},
      series={Progr. Math.},
      volume={282},
   publisher={Birkh\"{a}user Boston, Boston, MA},
       pages={219\ndash 243},
         url={https://doi.org/10.1007/978-0-8176-4934-0_9},
      review={\MR{2605171}},
}

\bib{Kuznet}{incollection}{
      author={Kuznetsov, Alexander},
       title={Derived categories view on rationality problems},
        date={2016},
   booktitle={Rationality problems in algebraic geometry},
      series={Lecture Notes in Math.},
      volume={2172},
   publisher={Springer, Cham},
       pages={67\ndash 104},
      review={\MR{3618666}},
}

\bib{laz09}{article}{
      author={Laza, Radu},
       title={The moduli space of cubic fourfolds},
        date={2009},
        ISSN={1056-3911},
     journal={J. Algebraic Geom.},
      volume={18},
      number={3},
       pages={511\ndash 545},
         url={https://doi.org/10.1090/S1056-3911-08-00506-7},
      review={\MR{2496456}},
}

\bib{Loo}{article}{
      author={Looijenga, Eduard},
       title={The period map for cubic fourfolds},
        date={2009},
        ISSN={0020-9910},
     journal={Invent. Math.},
      volume={177},
      number={1},
       pages={213\ndash 233},
         url={https://doi.org/10.1007/s00222-009-0178-6},
      review={\MR{2507640}},
}

\bib{markman}{incollection}{
      author={Markman, Eyal},
       title={A survey of {T}orelli and monodromy results for
  holomorphic-symplectic varieties},
        date={2011},
   booktitle={Complex and differential geometry},
      series={Springer Proc. Math.},
      volume={8},
   publisher={Springer, Heidelberg},
       pages={257\ndash 322},
         url={https://doi.org/10.1007/978-3-642-20300-8_15},
      review={\MR{2964480}},
}

\bib{Mertens}{article}{
      author={Mertens, Michael~H.},
       title={Automorphism groups of hyperbolic lattices},
        date={2014},
        ISSN={0021-8693,1090-266X},
     journal={J. Algebra},
      volume={408},
       pages={147\ndash 165},
         url={https://doi.org/10.1016/j.jalgebra.2013.09.034},
      review={\MR{3197177}},
}

\bib{Meachanetal}{article}{
      author={Meachan, Ciaran},
      author={Mongardi, Giovanni},
      author={Yoshioka, K\=ota},
       title={Derived equivalent {H}ilbert schemes of points on {K}3 surfaces
  which are not birational},
        date={2020},
        ISSN={0025-5874,1432-1823},
     journal={Math. Z.},
      volume={294},
      number={3-4},
       pages={871\ndash 880},
         url={https://doi.org/10.1007/s00209-019-02281-1},
      review={\MR{4074026}},
}

\bib{MR3423735}{article}{
      author={Mongardi, Giovanni},
       title={A note on the {K}\"{a}hler and {M}ori cones of hyperk\"{a}hler
  manifolds},
        date={2015},
        ISSN={1093-6106},
     journal={Asian J. Math.},
      volume={19},
      number={4},
       pages={583\ndash 591},
  url={https://doi-org.proxy.library.stonybrook.edu/10.4310/AJM.2015.v19.n4.a1},
      review={\MR{3423735}},
}

\bib{MS24}{misc}{
      author={Maulik, Davesh},
      author={Shen, Junliang},
      author={Yin, Qizheng},
      author={Zhang, Ruxuan},
       title={The d-equivalence conjecture for hyper-k\"ahler varieties via
  hyperholomorphic bundles},
        date={2024},
         url={https://arxiv.org/abs/2408.14775},
}

\bib{MV24}{misc}{
      author={Marquand, Lisa},
      author={Viktorova, Sasha},
       title={The defect of a cubic threefold},
        date={2024},
         url={https://arxiv.org/abs/2312.05118},
}

\bib{namikawa}{article}{
      author={Namikawa, Yoshinori},
       title={Mukai flops and derived categories},
        date={2003},
        ISSN={0075-4102,1435-5345},
     journal={J. Reine Angew. Math.},
      volume={560},
       pages={65\ndash 76},
         url={https://doi.org/10.1515/crll.2003.061},
      review={\MR{1992802}},
}

\bib{OG1}{article}{
      author={O'Grady, K.~G.},
       title={Involutions and linear systems on holomorphic symplectic
  manifolds},
        date={2005},
        ISSN={1016-443X,1420-8970},
     journal={Geom. Funct. Anal.},
      volume={15},
      number={6},
       pages={1223\ndash 1274},
         url={https://doi.org/10.1007/s00039-005-0538-3},
      review={\MR{2221247}},
}

\bib{OG2}{article}{
      author={O'Grady, Kieran~G.},
       title={Irreducible symplectic 4-folds and {E}isenbud-{P}opescu-{W}alter
  sextics},
        date={2006},
        ISSN={0012-7094,1547-7398},
     journal={Duke Math. J.},
      volume={134},
      number={1},
       pages={99\ndash 137},
         url={https://doi.org/10.1215/S0012-7094-06-13413-0},
      review={\MR{2239344}},
}

\bib{OG3}{article}{
      author={O'Grady, Kieran~G.},
       title={Irreducible symplectic 4-folds numerically equivalent to
  {$(K3)^{[2]}$}},
        date={2008},
        ISSN={0219-1997,1793-6683},
     journal={Commun. Contemp. Math.},
      volume={10},
      number={4},
       pages={553\ndash 608},
         url={https://doi.org/10.1142/S0219199708002909},
      review={\MR{2444848}},
}

\bib{OG16}{article}{
      author={O'Grady, Kieran~G.},
       title={Moduli of double {EPW}-sextics},
        date={2016},
        ISSN={0065-9266,1947-6221},
     journal={Mem. Amer. Math. Soc.},
      volume={240},
      number={1136},
       pages={ix+172},
         url={https://doi.org/10.1090/memo/1136},
      review={\MR{3460099}},
}

\bib{Ploog}{article}{
      author={Ploog, David},
       title={Equivariant autoequivalences for finite group actions},
        date={2007},
        ISSN={0001-8708,1090-2082},
     journal={Adv. Math.},
      volume={216},
      number={1},
       pages={62\ndash 74},
         url={https://doi.org/10.1016/j.aim.2007.05.002},
      review={\MR{2353249}},
}

\bib{voisinhodge}{article}{
      author={Voisin, Claire},
       title={Abel-{J}acobi map, integral {H}odge classes and decomposition of
  the diagonal},
        date={2013},
        ISSN={1056-3911,1534-7486},
     journal={J. Algebraic Geom.},
      volume={22},
      number={1},
       pages={141\ndash 174},
         url={https://doi.org/10.1090/S1056-3911-2012-00597-9},
      review={\MR{2993050}},
}

\bib{voisintorelli}{article}{
      author={Voisin, Claire},
       title={Th\'{e}or\`eme de {T}orelli pour les cubiques de {${\bf P}^5$}},
        date={1986},
        ISSN={0020-9910},
     journal={Invent. Math.},
      volume={86},
      number={3},
       pages={577\ndash 601},
         url={https://doi.org/10.1007/BF01389270},
      review={\MR{860684}},
}

\bib{yang2021lattice}{article}{
      author={Yang, Song},
      author={Yu, Xun},
       title={On lattice polarizable cubic fourfolds},
        date={2023},
        ISSN={2522-0144,2197-9847},
     journal={Res. Math. Sci.},
      volume={10},
      number={1},
       pages={Paper No. 2, 39},
         url={https://doi.org/10.1007/s40687-022-00368-6},
      review={\MR{4519216}},
}

\end{biblist}
\end{bibdiv}
\end{document}